\documentclass[smallextended,numbook]{svjour3}

\smartqed

\usepackage{amsfonts,amsmath}

\renewcommand*\makeheadbox{
  \hbox{\bfseries To be published in Journal of Theoretical Probability}
}

\begin{document}

\title{SDEs driven by a time-changed L\'evy process and their associated time-fractional
        order pseudo-differential equations}

\author{Marjorie Hahn \and Kei Kobayashi \and\\ Sabir Umarov}

\institute{M. Hahn \at
              Department of Mathematics, Tufts University, 503 Boston Avenue, Medford, MA 02155\\
              Tel.: 1-617-627-2363 ~~~~ Fax:  1-617-627-3966 \\
              \email{marjorie.hahn@tufts.edu}           
           \and
           K. Kobayashi  \at
              Department of Mathematics, Tufts University, Medford, MA 02155\\
              \email{kei.kobayashi@tufts.edu}
           \and
            S.Umarov  \at
              Department of Mathematics, Tufts University, Medford, MA 02155\\
              \email{sabir.umarov@tufts.edu}
}

\date{Submitted: January 17, 2010; Revised: April 6, 2010 
}

\maketitle

\begin{abstract}
It is known that the transition probabilities of a solution to a
classical It\^o stochastic differential equation (SDE) satisfy
in the weak sense the associated Kolmogorov equation.
The
Kolmogorov equation is a 
partial differential equation
with coefficients determined by the corresponding SDE.
Time-fractional Kolmogorov type equations are used to model complex
processes in many fields. However, the class of SDEs that is
associated with these equations is unknown except in a few special
cases. The present paper shows that in the cases of either
time-fractional order or more general time-distributed order
differential equations, the associated class of SDEs can be
described within the framework of SDEs driven by semimartingales.
These semimartingales are time-changed L\'evy processes where the
independent time-change is given respectively by the inverse of a
single or mixture of independent stable subordinators. Examples are
provided, including a fractional analogue of the Feynman-Kac
formula.
\end{abstract}

\section{Introduction}
\label{intro}

In the last few decades, fractional Kolmogorov or equivalently
fractional\break 
Fokker-Planck
equations have appeared as an essential
tool for the study of dynamics of various complex stochastic
processes arising in anomalous diffusion in physics
\cite{MetzlerKlafter00,Zaslavski}, finance \cite{GMSR}, hydrology
\cite{BWM}, and cell biology \cite{Saxton}. Complexity includes
phenomena such as the presence of weak or strong correlations,
different sub- or super-diffusive modes, and jump effects.
For example,
experimental studies of the motion of
macromolecules in a cell membrane
show apparent subdiffusive
motion with several simultaneous diffusive modes
(see \cite{Saxton}).

The present paper identifies a wide class of stochastic differential
equations (SDEs) whose associated partial differential equations are
represented by time-fractional order pseudo-differential equations.
This connection provides stochastic processes whose dynamics
correspond to time-fractional order pseudo-differential equations.

Let $B_t$ be an $m$-dimensional Brownian motion defined on
a probability space $(\Omega,\mathcal{F},\mathbb{P})$ with a
complete right-continuous filtration $(\mathcal{F}_t)$.
A deep connection
between a stochastic process and
its associated partial differential equation is
expressed through
the Kolmogorov forward and backward equations \cite{Applebaum}. This concept is based
on the relationship between two main components:
(i) the Cauchy problem
\begin{align}
\label{first}
\frac{\partial u(t,x)}{\partial t}=\mathcal{A}
u(t,x), \ u(0,x)=\varphi(x), \ t>0, \, x
\in \mathbb{R}^n,
\end{align}
where $\mathcal{A}$ is the differential operator
\begin{equation} \label{operatorA}
\mathcal{A}= \sum_{j=1}^n b_j(x) \frac{\partial }{\partial x_j} +
\frac 12 \sum_{i,j=1}^n \sigma_{i,j}(x)\frac{\partial^2 }{\partial
x_i
\partial x_j},
\end{equation}
with coefficients $b_j(x)$ and $\sigma_{i,j}(x)$ satisfying some regularity conditions; and (ii)
 the associated class of It\^o SDEs given by
\begin{equation}
\label{sdegen} dX_t=b(X_t) dt+\sigma(X_t)dB_t, \, X_{0}=x.
\end{equation}
The coefficients of SDE \eqref{sdegen}
are connected with the coefficients
of the operator $\mathcal{A}$ as follows:
$b(x)=(b_1(x),\ldots,b_n(x))$ and $\sigma_{i,j}(x)$ is the
$(i,j)$-th entry of the product of the $n\times m$ matrix $\sigma
(x) $ with its transpose $\sigma^T(x)$.

One mechanism for establishing this relationship is via semigroup
theory, in which the operator $\mathcal{A}$ is recognized as the
infinitesimal generator of the semigroup
$T_t(\cdot)(x):=\mathbb{E}[{(\cdot)(X_t)}|X_0=x]$ (defined, for
instance, on the Banach space $C_0(\mathbb{R}^n)$ with supnorm),
i.e. $\mathcal{A}\varphi(x)=\lim_{t\to 0}{(T_t-I)\varphi(x)}/{t},$
$\varphi \in {\rm Dom}(\mathcal{A}),$ the domain of $\mathcal{A}$. A
unique solution to
(\ref{first})
is represented by $u(t,x)=(T_t\varphi)(x)$.

Enlarging the
class of
SDEs in (\ref{sdegen}) to those driven by a L\'evy
process leads to a generalization of
connection (i)--(ii)
where
the analogous operator on the right-hand side of (\ref{first}) has
additional terms corresponding to jump components of the driving
process (see  \cite{Applebaum,Situ} and references therein). In this
case, the operator $\mathcal{A}$ in (\ref{operatorA}) takes the form
$\mathcal{L}(x,\mathbf{D}_x)$ in \eqref{levyPsdo}.

A fractional generalization of the Cauchy problem
(\ref{first})
with
$\mathcal{A}=\mathcal{L}(x,\mathbf{D}_x)$, in the sense that the
first order time derivative on the left side of equation
(\ref{first}) is replaced by a time-fractional order derivative, has
appeared in the framework of continuous time random walks (CTRWs)
and fractional kinetic theory
\cite{GillisWeiss,GMV,MeerschaertScheffler,MetzlerKlafter00,MontrollScher,Zaslavski}.
Papers \cite{GM1,Meeschaert1,UG,US}
establish that time-fractional versions of the Cauchy problem are
connected with limit processes arising from certain weakly
convergent sequences or triangular arrays of CTRWs. These limit
processes are time-changed L\'evy processes, where the time-change
arises as the first hitting time of level $t$ (equivalently, the
inverse) for a single stable subordinator.

This paper generalizes the
connection (i)--(ii) to 
time-fractional pseudo-differential equations
that imply a
fractional analogue of Kolmogorov equations. First, it establishes
the class of SDEs replacing \eqref{sdegen} which is associated with
the following Cauchy problem:
\begin{align}
\label{CP1} \tag{1.1$^{'}$}
\mathbf{D}^{\beta}_*u(t,x)=\mathcal{L}(x,\mathbf{D}_x) u(t,x), \ u(0,x)=\varphi(x), \ t>0, \, x \in \mathbb{R}^n,
 \end{align}
where $\mathbf{D}^{\beta}_*$ is the fractional derivative in the sense
of Caputo with $\beta\in (0,1)$ (see Section \ref{sec
PRELIMINARIES}), and $\mathcal{L}(x,\mathbf{D}_x)$ is the
pseudo-differential operator in \eqref{levyPsdo}.
The driving processes of the associated class of SDEs are L\'evy
processes composed with the inverse of a $\beta$-stable
subordinator, $\beta\in (0,1)$ (Theorem \ref{multi} for $N=\nolinebreak 1$).
Since such processes are semimartingales, SDEs with respect to them
are meaningful and have the form
in \eqref{sdelevyfr}.
A partial result when the driving process is either Brownian or L\'evy
stable motion with drift time-changed by the inverse of a single
stable subordinator is considered in
\cite{MagdziarzWeron,MagdziarzWeronKlafter} without specifying the
explicit form of the corresponding SDEs.

More generally, the class of SDEs in the above discussion when the time-change process
is the inverse of an arbitrary mixture of independent stable
subordinators gives rise to a Cauchy problem with a fractional
derivative with distributed orders (see \eqref{dode}) on the left of
\eqref{CP1}, namely,
$$\mathcal{D}_\mu u(t,x)=\mathcal{L}(x,\mathbf{D}_x) u(t,x).$$
In this case, the time-change process is no longer the inverse of a
stable subordinator if at least two different indices arise in the
mixture. Moreover, SDEs corresponding to time-fractional Kolmogorov
equations cannot be described within the classical Brownian- or
L\'evy-driven SDEs.

Section \ref{sec PRELIMINARIES} of this paper recalls the required auxiliary facts.
Section \ref{sec RESULTS} formulates and proves the main results of
the paper, and
provides examples, including a fractional analogue of the
Feynman-Kac formula. Section \ref{sec ALTANATIVE} illustrates an alternative technique for establishing the main results in the case of Brownian motion.

\section{Preliminaries and auxiliaries} \label{sec PRELIMINARIES}

The \emph{fractional integral of order} $\beta>0$ is
\begin{equation}\label{frac-int}
J^{\beta}g(t)=\frac{1}{\Gamma(\beta)}\int_{0}^{t} (t-u)^{\beta
-1}g(u)du,\ t>0,
\end{equation}
where $\Gamma(\cdot)$ is Euler's gamma function. By convention, $J^0
=I,$ the identity operator, and  $Jg(t):= J^1 g(t)$, the integration
operator.
The \emph{fractional derivative of order $\beta \in (0,1)$ in the
sense of Caputo} is $\mathbf{D}_{\ast}^{\beta}g(t)=J^{1-\beta} \frac {d}{dt} g(t)$, $t>0$.  By convention, set
$\mathbf{D}_{\ast}^{\beta}=\frac{d}{dt}$ for
$\beta=1.$
The Laplace transform of $\mathbf{D}_{\ast}^{\beta}g$ is
(\cite{GM97})
\begin{align}
\label{laplace} \widetilde{[\mathbf{D}_{\ast}^{\beta}g]}(s) =
s^{\beta} \tilde{g}(s) - s^{\beta -1}g(0+),
\end{align}
where $\tilde{g}(s) \equiv
{\mathcal{L}}[g](s)=\int_0^{\infty}g(t)e^{-st}dt,$ the Laplace
transform of $g.$

Let $\mu$ be a finite measure on $[0,1].$ The
\emph{fractional derivative with distributed orders} is the operator
(see, e.g.,\ \cite{UG05})
\begin{equation}
\label{dode} \mathcal{D}_{\mu} g(t)=\int_0^1
\mathbf{D}_{\ast}^{\beta} g(t)\, d\mu (\beta).
\end{equation}
These operators provide a generalization of fractional order
derivatives. The mapping $\beta \mapsto \mathbf{D}_{\ast}^{\beta}
g(t)$ is continuous on $[0,1)$ for a differentiable function $g$.
For example, 
if $\mu = C_1
\delta_{\beta_1} + C_2 \delta_{\beta_2}$,
then $\mathcal{D}_{\mu}
g(t)= C_1 \mathbf{D}_{\ast}^{\beta_1}g(t)+ C_2
\mathbf{D}_{\ast}^{\beta_2}g(t).$

If a function $\psi(x,\xi) :  \mathbb{R}^n \times \mathbb{R}^n  \to
\mathbb{C} $ is continuous and satisfies a suitable growth condition
as $|\xi| \to \infty$, then for 
$u \in C^{\infty}_0(\mathbb{R}^n)$, the operator
\begin{equation}
\mathcal{A} u(x)=\frac{1}{(2
\pi)^n}\int_{\mathbb{R}^n}\psi(x,\xi)\hat{u}(\xi)e^{i(x,\xi)}d\xi, \
x\in\mathbb{R}^n, \label{PSO-def}
\end{equation}
where $\hat{u}(\xi)=\int_{\mathbb{R}^n}u(x)e^{-i(x,\xi)} dx$, is meaningful and 
called a {\it pseudo-differential operator} with
{\it symbol} $\psi(x,\xi)$.
For properties of pseudo-differential operators, see the monographs \cite{Hermander,Jacob,Taylor}.

Pseudo-differential operators of interest in this paper are
infinitesimal generators of strongly continuous semigroups
constructed from stochastic processes which are solutions to SDEs
driven by a L\'evy process. Such processes are Feller processes, and
therefore, they have strongly continuous semigroups (see
\cite{Applebaum}). To this end, we will consider symbols
$\psi(x,\xi)$ which are continuous in $x \in \mathbb{R}^n$ and $\xi
\in \mathbb{R}^n$ and, for each fixed $x$, both hermitian and conditionally
positive definite in $\xi$
(see \cite{Applebaum} for details).

A {\it L\'evy process} $L_t \in \mathbb{R}^n, ~ t\geq 0$, with $
L_0=0,$ is an
adapted c\`adl\`ag process 
with independent stationary increments
such that for all $\epsilon, \,  t >0$, 
$\lim_{s\to t}
\mathbb{P}(|L_t-L_s|>\epsilon)=0.$ 
L\'evy processes are
characterized by three parameters: a vector $b \in \mathbb{R}^n,$ a
nonnegative definite matrix $\Sigma,$ and a measure $\nu$ defined on
$\mathbb{R}^n \setminus \{0\}$ such that $\int \min (1,|x|^2) d \nu
<\infty$, called its L\'evy measure. The L\'evy-Khintchine formula
characterizes a L\'evy process (as an infinitely divisible process)
in terms of its characteristic function $\Phi_t(\xi)=e^{t \Psi
(\xi)}$, with
\begin{equation}
\label{levy-khintchin} \Psi(\xi) = i(b,\xi) - \frac{1}{2} (\Sigma
\xi, \xi) + \int_{\mathbb{R}^n \setminus \{0\}}
\hspace{-1mm}(e^{i(w, \xi)}- 1-i (w,\xi) \chi_{(|w| \le
1)}(w))\nu(dw).
\end{equation}
The function $\Psi$ is called the \textit{L\'evy symbol} of $L_t$
(see, e.g.\ \cite{Applebaum,Sato}.)

Particularly important for this paper is the
class of stable subordinators.
For $\beta \in (0,1)$, a
\textit{$\beta$-stable subordinator} is
a one-dimensional strictly increasing L\'evy process $D_t$ starting at $0$
which is self-similar, i.e.\ $\{D_{at}, t\ge 0\}$ has the same finite-dimensional
distributions as $\{a^{1 /
\beta}D_t, t\ge 0\}$,
and
the Laplace transform for $D_1$  is given by
\begin{equation}
\label{subordinator} \mathbb{E}[e^{-sD_1}]= e^{-s^{\beta}}, \ s \ge
0.
\end{equation}
It follows from the general theory of Laplace transforms (see, e.g.
\cite{Widder}) that the density $f_{D_{_1}}(\tau)$ of $D_1$ is
infinitely differentiable on $(0, \infty),$ with the following
asymptotics at zero and infinity
\cite{MinardiLuchkoPagnini,UzhaykinZolotarev}:
\begin{align}
&f_{{D_{_1}}}(\tau) \sim \frac{({\frac \beta
\tau})^{\frac{2-\beta}{2(1-\beta)}}}{\sqrt{2\pi \beta (1-\beta)}} \,
e^{-(1-\beta)({\frac \tau  \beta})^{-\frac{\beta}{1-\beta}}}, \,
\tau
\to 0; \label{atzero}\\
&f_{D_{_1}}(\tau) \sim \frac{\beta}{\Gamma(1-\beta) \tau^{1+\beta}},
\, \tau \to \infty. \label{atinfinity}
\end{align}

Consider an SDE driven by a L\'evy process
\begin{align} \label{sdeLevy}
     Y_t&=x + \int_0^t b(Y_{s-}) ds + \int_0^t \sigma(Y_{s-}) dB_s\\
          & \ \ \ + \int_0^t \int_{|w|<1} \hspace{-1mm}H(Y_{s-},w)\tilde{N}(ds, dw)+ \int_0^t \int_{|w|\geq 1} \hspace{-1mm}K(Y_{s-},w) N(ds, dw),\notag
\end{align}
where $x \in \mathbb{R}^n$, and
the continuous mappings $b:\mathbb{R}^n \to
\mathbb{R}^n,$ $\sigma: \mathbb{R}^n\to \mathbb{R}^{n \times m},$
$H:\mathbb{R}^n \times \mathbb{R}^n \to \mathbb{R}^n$, and
$K:\mathbb{R}^n\times \mathbb{R}^n \to \mathbb{R}^n$ satisfy the
following
Lipschitz and growth
conditions: there exist positive constants $C_1$ and
$C_2$ satisfying
\begin{align}  \label{lip}
\bullet \ \ &|b(y_1)-b(y_2)|^2+\|\sigma(y_1)-\sigma(y_2)\|^2 +\int_{|w|<1}|H(y_1,w)-H(y_2,w)|^2\nu(dw) \notag \\ 
&\le C_1|y_1-y_2|^2, \ \forall \, y_1, y_2 \in \mathbb{R}^n; \\
\label{growth}
\bullet \ \ &\int_{|w|<1}|H(y,w)|^2\nu(dw)  \le C_2(1+|y|^2),  \  \forall \, y
\in \mathbb{R}^n.
\end{align}
Under these conditions, SDE \eqref{sdeLevy} has a unique strong
solution $Y_t$ (see, \cite{Applebaum,Situ}).
If the coefficients
$b$, $\sigma$, $H$, and $K$ are bounded, then $(T_t\varphi)(x)=
\mathbb{E}[\varphi(Y_{t})|Y_0= x]$ is a strongly continuous
contraction semigroup defined on the Banach space
$C_0(\mathbb{R}^n).$  Moreover, the pseudo-differential equation
associated with the process $Y_t$ takes the form
\begin{equation}
\label{levyFP} \displaystyle{ \frac{\partial u(t,x)}{\partial t}  =
\mathcal{L}(x,\mathbf{D}_x) u(t,x)},
\end{equation}
where the infinitesimal generator $\mathcal{L}(x,\mathbf{D}_x)$
is a pseudo-differential operator
 with the symbol
\begin{align}
\label{levysymbol}\Psi(x,\xi)
     &= i(b(x),\xi) - \frac{1}{2} (\Sigma (x) \xi, \xi) \\
     & \ \ \ + \int_{\mathbb{R}^n \setminus \{0\}} \hspace{-2mm} (e^{i(G(x,w),\xi)}- 1-i(G(x,w), \xi) \chi_{(|w| < 1)}(w))\nu(dw),\notag
\end{align}
where $G(x,w)=H(x,w)$ if $|w|<1,$ and $G(x,w)=K(x,w)$ if $|w|\ge 1$\linebreak
(\cite{Applebaum,Situ}). For each fixed $x\in \mathbb{R}^n$, the
symbol $\Psi(x,\xi)$ is continuous, hermitian, and conditionally
positive definite \cite{Courrege,Jacob}. Using
$\mathbf{D}_x=\frac{1}{i}(\partial/\partial x_1,\ldots,\partial/
\partial x_n)$, the pseudo-differential operator $\mathcal{L}(x,\mathbf{D}_x)$ has
the form
\begin{align} \label{levyPsdo}
    &\mathcal{L}(x,\mathbf{D}_x)\varphi(x)
    = i(b(x),\mathbf{D}_x)\varphi(x) - {\frac 1 2}(\Sigma(x) \mathbf{D}_x, \mathbf{D}_x)\varphi(x) \\
&+ \int_{\mathbb{R}^n \setminus \{ 0\}} \hspace{-2mm}
\bigl[\varphi(x+G(x,w)) -\varphi(x)- i\chi_{(|w|<1)}(w) (G(x,w),
\mathbf{D}_x)\varphi(x) \bigr] \nu(dw).\notag
\end{align}
Here, $\mathcal{L}(x,\mathbf{D}_x):C_0^2(\mathbb{R}^n) 
\longrightarrow C_0(\mathbb{R}^n),$ i.e. $C_0^2(\mathbb{R}^n)
\subset {\rm Dom}\bigl(\mathcal{L}(x,\mathbf{D}_x)\bigr).$

\section{Main results and examples} \label{sec RESULTS}

Let $D_t$ be an $(\mathcal{F}_t)$-adapted strictly increasing
c\`adl\`ag process starting at $0$ such that
$\lim_{t\to\infty}D_t=\infty$ a.s.  The
\textit{inverse} or the \textit{first hitting time process} $E_t$ of
$D_t$ is defined by $E_t:=\inf \{ \tau \ge 0: D_{\tau}>t \}$.
The inverse $E_t$ is a \textit{continuous $(\mathcal{F}_t)$-time-change},
i.e.\ it is a continuous, nondecreasing family of $(\mathcal{F}_t)$-stopping times.
In fact, for any $t, \tau>0$, we have
$\{E_t<\tau\}=\{D_{\tau-}>t\}\in\mathcal{F}_\tau$. Hence, by the
right-continuity of the filtration $(\mathcal{F}_t)$, each random
variable $E_t$ is an $(\mathcal{F}_t)$-stopping time.

Let $D_{1, t}$ and $D_{2, t}$ be independent $(\mathcal{F}_t)$-adapted
strictly increasing c\`adl\`ag processes.   Then $D_t=D_{1, t}+D_{2, t}$ also possesses
the same property, and its inverse process $E_t$ satisfies
$\mathbb{P}(E_t \le  \tau)=\mathbb{P}(D_{\tau}> t)=1-(F_{\tau}^{(1)} \ast
F_{\tau}^{(2)})(t).$
Here,
for $k=1,2$, $F_{\tau}^{(k)}(t)=$
$\mathbb{P}(D_{k,\tau} \le t)$ with density $f_{\tau}^{(k)}$ (if it exists), and
$\ast$ denotes convolution of cumulative distribution functions or
densities, whichever is required.  For notational convenience, if
$a, b > 0$, let
\begin{align*}
&\left[F_1^{(1)}\left( \frac \cdot a \right)\ast
F^{(2)}_1\left(\frac \cdot b\right)\right](t) :=\int_{s=0}^{s=t}
F_1^{(1)}\left(\frac{t-s}{a}\right)
dF_1^{(2)}\left(\frac{s}{b}\right),\\
\intertext{which, if the density functions exist, can also be written
as}
&\left[F_1^{(1)}\left( \frac \cdot a \right)\ast
F^{(2)}_1\left(\frac \cdot b\right)\right](t) =\frac{1}{b}
\int_{s=0}^{s=t} \bigl(Jf_1^{(1)}\bigr)\left(\frac{t-s}{a}\right)
f_1^{(2)}\left(\frac{s}{b}\right) ds,
\end{align*}
where $J$ is the usual integration operator.

\begin{lemma}
\label{11} Let $D_t=c_1 D_{1, t}+ c_2 D_{2, t}$, where $c_1, c_2$
are positive constants and $D_{1, t}$ and $D_{2, t}$ are independent
stable subordinators with respective indices $\beta_1$ and $\beta_2$
in $(0,1)$. Then the inverse $E_t$ of $D_t$ satisfies
 \begin{equation}
\mathbb{P}(E_t \le \tau)= 1- \left[F_1^{(1)}\left(\frac{\cdot}{c_1
\tau^{1/
 \beta_1}}\right) \ast F_1^{(2)}\left(\frac{\cdot}{c_2 \tau^{1/
\beta_2}}\right)\right](t) \label{11conv}
\end{equation}
and has the density
\begin{align}
\label{density10} f_{E_t}(\tau) =-  \frac{\partial}{\partial
\tau}\Bigg\{\frac{1}{c_2 \tau^{1/ \beta_2}}
\left[\bigl(Jf_1^{(1)}\bigr)\left(\frac{\cdot}{c_1 \tau^{1/
\beta_1}}\right)
   \ast f_1^{(2)}\left(\frac{\cdot}{c_2 \tau^{1/ \beta_2}}\right)\right](t)\Bigg\}.
\end{align}
\end{lemma}

\begin{proof}
Since $D_{1, \tau}$ and $D_{2, \tau}$ are independent and
self-similar processes,
\begin{align*}
\mathbb{P}(E_t \le {\tau} )&=\mathbb{P}(D_{\tau} >t)
=1-\mathbb{P}\bigl(c_1 \tau^{1/ \beta_1}D_{1,1}+c_2 \tau^{1/ \beta_2} D_{2,1} \le t\bigr),
\end{align*}
by which \eqref{11conv} follows.  Differentiating \eqref{11conv} with respect to $\tau$ yields \eqref{density10}.  \qed
\end{proof}

\begin{lemma} \label{estimation}
For any $t<\infty$, the density $f_{E_t}(\tau)$ in
\eqref{density10} is bounded, and there exist a number $\beta \in
(0,1)$ and positive constants $C$, $k$, not depending on $\tau$,
such that
\begin{equation}
\label{est} f_{E_t}(\tau) \le C \exp
\Bigl(-k\tau^{\frac{1}{1-\beta}}\Bigr)
\end{equation}
for $\tau$ large enough.
\end{lemma}

Routine elementary calculations using \eqref{atzero} and \eqref{atinfinity} yield Lemma \ref{estimation}.
This lemma can be extended for processes $E_t$ which are
inverses of stochastic processes of the form $D_t=\sum_{k=1}^N c_k
D_{k, t},$ where $D_{k, t}, k=1,\ldots,N,$ are independent stable
subordinators of respective indices $\beta_k\in(0,1)$, and $c_k$ are
positive constants.

Theorems \ref{part2} and \ref{continuum} require the following {\it
assumption}: $\{T_t, t \ge 0\}$ is a strongly continuous semigroup
defined on a Banach space $\mathcal{X}$ with norm $\|\cdot\|$, such
that the estimate
\begin{equation} \label{semigrestimate}
\|T_t\varphi\| \le M \|\varphi\|e^{\omega t}
\end{equation}
is valid for some constants $M >0$ and $\omega \ge 0.$ For example,
if $\{T_t\}$ is a semigroup associated with a Feller process, then
\eqref{semigrestimate} is satisfied with $\omega=0$. This assumption
implies that  any number $s$ with $Re(s)>\omega$ belongs to the
resolvent set $\rho(\mathcal{A})$ of the infinitesimal generator
$\mathcal{A}$ of $T_t$ and the resolvent operator is represented in
the form $R(s,\mathcal{A})=\int_0^{\infty}e^{-st}T_t\, dt$ (see
\cite{EngelNagel}).

\begin{theorem} \label{part2}
Define the process $D_t=\sum_{k=1}^N c_k D_{k,t},$ where $D_{k,t}$,
$1\le k\le N,$ are independent stable subordinators with respective
indices $\beta_k \in (0,1)$ and constants $c_k>0.$ Let $E_t$ be the
inverse process to $D_t$. Suppose that $T_t$ is a strongly continuous
semigroup in a Banach space $\mathcal{X}$, satisfies
(\ref{semigrestimate}), and has infinitesimal generator
$\mathcal{A}$. 
Then, for each fixed $t\ge 0$, the integral
$\int_0^{\infty}f_{E_t}(\tau)T_{\tau} \varphi \, d\tau$ exists and
the vector-function $v(t)=\int_0^{\infty}f_{E_t}(\tau)T_{\tau}
\varphi \, d\tau,$ where $\varphi \in \mbox{Dom}(\mathcal{A}),$
satisfies the abstract Cauchy problem
\begin{align} \label{abstract22222}
\sum_{k=1}^N C_k \mathbf{D}_{\ast}^{\beta_k}v(t) =\mathcal{A} v(t), \, t>0,
\ v(0)=\varphi,
\end{align}
where $C_k=c_k^{{\beta_k}},$ $k=1,\ldots,N.$
\end{theorem}

\begin{proof}
For simplicity, the proof will be given in the case $N=2$. First,
define the vector-function $p(\tau)=T_{\tau} \varphi,$ where $\varphi
\in \mbox{Dom}(\mathcal{A}).$ In accordance with the conditions of
the theorem, $p(\tau)$ satisfies the abstract Cauchy problem
\begin{equation}
\label{abstract20} \displaystyle{ \frac{\partial p(\tau)}{\partial
\tau} = \mathcal{A} p(\tau)}, ~~ p(0) = \varphi,
\end{equation}
where $\mathcal{A}$ is the infinitesimal generator of
$T_{\tau}.$ Now consider the integral $\int_0^{\infty}f_{E_t}(\tau)
T_{\tau} \varphi\, d\tau.$ It follows from Lemma \ref{estimation}
and condition (\ref{semigrestimate}) that
\begin{align}\label{estimate000}
\int_0^{\infty} f_{E_t}(\tau) \|T_{\tau} \varphi\|
\, d\tau \le C \| \varphi \| \int_0^{\infty} e^{-(k
\tau^{\frac{1}{1-\beta}}-\omega \tau)}\, d\tau < \infty,
\end{align}
where $\beta\in(0,1)$ and $C$, $k>0$ are constants. Hence,
the Bochner
integral $\int_0^{\infty} f_{E_t}(\tau) T_{\tau} \varphi\, d\tau$ exists
for each fixed $t \ge 0.$
Denote this
vector-function by
\begin{equation}
\label{def_v(t)} v(t):=\int_0^{\infty} f_{E_t}(\tau) T_{\tau}
\varphi\, d\tau.
\end{equation}
The definition of $T_t$ implies
$v(0)=\lim_{t \to 0+} \int_0^{\infty}f_{E_t}(\tau)T_{\tau} \varphi \,
d\tau = T_{0} \varphi = \varphi,$
in the norm of $\mathcal{X}.$ By \eqref{density10},
\[
v(t)= -\int_0^{\infty} \hspace{-0.5mm} \frac{\partial}{\partial \tau}\Bigg\{
\frac{1}{c_2 \tau^{1/ \beta_2}}
\left[(Jf_1^{(1)})\left(\frac{\cdot}{c_1 \tau^{1/
\beta_1}}\right) \ast f_1^{(2)}\left(\frac{\cdot}{c_2 \tau^{1/
\beta_2}}\right)\right](t)\Bigg\} \hspace{0.5mm}T_{\tau} \varphi \,
d \tau.
\]
Since
\begin{align*}
{\mathcal{L}} \left[\frac{1}{b}\hspace{1mm}\bigl(Jf_1^{(1)}\bigr)
\left( \frac t a \right) \ast f^{(2)}_1\left(\frac t
b\right)\right](s)&= \frac{1}{b} \frac{1}{as} \left(a
\widetilde{f_1^{(1)}} (as) \right) \left(b \widetilde{f_1^{(2)}}
(bs) \right)\\& =\frac{1}{s} \widetilde{f_1^{(1)}} (as)
\widetilde{f_1^{(2)}} (bs),
\end{align*}
using \eqref{subordinator}, the Laplace transform of $v(t)$ takes
the form
\begin{align}
\label{laplace_v(t)} \tilde{v}(s)
&=-\int_0^{\infty}
\frac{\partial}{\partial \tau} \Big\{ \frac{1}{s} e^{-\tau
c_1^{\beta_1} s^{\beta_1}} e^{-\tau c_2^{\beta_2} s^{\beta_2}}
\Big\} \hspace{0.5mm}T_{\tau} \varphi
\, d \tau\\
&=(c_1^{\beta_1} s^{\beta_1-1}+ c_2^{\beta_2} s^{\beta_2-1}) \, \widetilde{[T_{\tau} \varphi]}(c_1^{\beta_1}
s^{\beta_1}+c_2^{\beta_2} s^{\beta_2})\notag \\
             &=(C_1 s^{\beta_1-1}+ C_2 s^{\beta_2-1}) \, \tilde{p} \, (C_1
s^{\beta_1}+C_2 s^{\beta_2}),\notag
\end{align}
where $C_k=c_k^{\beta_k}$, $k=1,2.$
Due to \eqref{semigrestimate}, this is well defined for all $s$ such that $C_1 s^{\beta_1}+C_2 s^{\beta_2} > \omega$.
On the
other hand it follows from (\ref{abstract20}) that
\begin{align} \label{lapforp}
(s-\mathcal{A})\tilde{p}(s)=\varphi, ~~ \forall s >\omega.
\end{align}
Let $\omega_0 \ge 0$ be a number such that $s > \omega_0$ iff $ C_1
s^{\beta_1}+C_2 s^{\beta_2} > \omega.$ Then \eqref{laplace_v(t)} and
\eqref{lapforp} together yield
\[
[C_1s^{\beta_1}+ C_2 s^{\beta_2} - \mathcal{A}]
\tilde{v}(s)=(C_1s^{\beta_1-1}+ C_2 s^{\beta_2-1})\varphi, ~~  s >
\omega_0.
\]
Writing this in the form
\[
C_1[s^{\beta_1}\tilde{v}(s)- s^{\beta_1-1}v(0)]
 + C_2[s^{\beta_2}\tilde{v}(s)-s^{\beta_2-1}v(0)]
 = \mathcal{A} \tilde{v}(s), ~~  s > \omega_0,
\]
recalling \eqref{laplace}, and applying the inverse Laplace
transform to both sides gives
\[
C_1 \mathbf{D}_{\ast}^{\beta_1}v(t)+C_2
\mathbf{D}_{\ast}^{\beta_2}v(t) =\mathcal{A} v(t).
\]
Hence $v(t)$ satisfies the Cauchy problem
(\ref{abstract22222}).
The case $N > 2$ can
be proved using the same method.  \qed
\end{proof}

The next theorem provides an extension with $D_t$ as the weighted
average of an arbitrary number of independent stable subordinators.
It is easy to verify that the process $D_t=\sum_{k=1}^Nc_k D_{k,t}$
given in Theorem {\ref{part2} satisfies
\begin{align} \label{llsymbol1}
\ln \mathbb{E} \bigl[e^{-s D_t}\bigr]\Big|_{t=1}=-\sum_{k=1}^N
c_k^{\beta_k}s^{\beta_k}, ~ s \ge 0.
\end{align}
The function on the right-hand side of \eqref{llsymbol1} can be
expressed as the integral $-\int_0^1 s^{\beta} d\mu(\beta),$ with
$\mu$ the finite atomic measure
$\mu = \sum_{k=1}^N
c_k^{\beta_k}\delta_{\beta_k}.$
 The integral $\int_0^1
s^{\beta} d\mu(\beta)$ is meaningful for any finite measure $\mu$
defined on $[0,1].$ Let
 $\mathcal{S}$ designate the class of 
c\`adl\`ag $(\mathcal{F}_t)$-adapted strictly increasing
processes $V_t$ whose Laplace transform is
given by
\begin{align} \label{llsymbol3}
  \mathbb{E} \bigl[e^{-s V_t}\bigr]= \exp\Bigl[-t \int_0^1 s^{\beta} d \mu
(\beta)\Bigr], ~ s \ge 0,
\end{align}
where $\mu$ is a finite measure on $[0,1].$
By construction, $V_0=0$
a.s., and $V_t$ can be considered as a weighted mixture of
independent stable subordinators. For the process $V_t \in \cal{S}$
corresponding to a finite measure $\mu$, we use the notation
$V_t=D(\mu;t)$ to indicate this correspondence.

\begin{theorem} \label{continuum}
Assume that $D(\mu;t) \in \mathcal{S}$ where $\mu$ is a positive
finite measure with $supp \, \mu \subset (0,1),$  and let $E_t$ be
the inverse process to $D(\mu;t)$. Then the vector-function
$v(t)=\int_0^{\infty}f_{E_t}(\tau)T_{\tau} \varphi \, d\tau,$ where
$T_{t}$ and $\varphi$ are as in Theorem \ref{part2}, exists and
satisfies the abstract Cauchy problem
\begin{align} \label{abstract11111}
\mathcal{D}_{\mu}v(t):=\int_{0}^1  \mathbf{D}_{\ast}^{\beta}v(t)
d\mu(\beta) =\mathcal{A}v(t), \, t>0, \ v(0)=\varphi.
\end{align}
\end{theorem}

\begin{proof} We briefly sketch the proof since the idea is
similar to the proof of Theorem \ref{part2}. Since $supp \, \mu
\subset (0,1)$, the density $\displaystyle f_{{D(\mu;t)}}(\tau),
\tau \ge 0,$ exists and has asymptotics (\ref{atzero}) with some
$\beta=\beta_0\in (0,1)$ and (\ref{atinfinity}) with some
$\beta=\beta_1\in (0,1)$.
 This implies the
existence of the vector-function $v(t).$ Further, one can readily
see that
$v(t)= -  \int_0^{\infty}
\frac{\partial}{\partial \tau} \{J f_{{D(\mu;\tau)}}(t)\}
\hspace{0.3mm}(T_{\tau}\varphi) d\tau.$ Now it follows from the
definition of $D(\mu;t)$ that the Laplace transform of $v(t)$
satisfies
\begin{align} \label{lapgen}
\tilde{v}(s) = \frac{\int_0^1 s^{\beta} d\mu(\beta)}{s}
\int_0^{\infty} e^{-\tau \int_0^1
s^{\beta}d\mu(\beta)}(T_{\tau}\varphi)d\tau = \frac{\eta(s)}{s}
\hspace{0.5mm} \tilde{p}(\eta(s)), ~ s
> \bar{\omega},
\end{align}
where $\eta(s)=\int_0^1 s^{\beta}d\mu(\beta)$, $p$ is a solution to
the abstract Cauchy problem (\ref{abstract20}), and $\bar{\omega}
>0$ is a number such that $s > \bar{\omega}$ if $\eta(s) > \omega$
($\bar{\omega}$ is uniquely defined since $\eta(s)$ is a strictly
increasing function). Combining (\ref{lapgen}) and (\ref{lapforp}),
\begin{equation}\label{new}
(\eta(s)-\mathcal{A}) \tilde{v}(s)=\varphi
\hspace{0.2mm}\frac{\eta(s)}{s}, ~ s > \bar{\omega}.
\end{equation}
Applying the Laplace transform to (\ref{abstract11111}) yields
\eqref{new}, as desired. \qed
\end{proof}

\begin{remark}
\mbox{ }

\hspace{-1.5mm}a) Equation \eqref{abstract11111} is called a \textit{distributed
order differential equation (DODE)}.

\hspace{-1.5mm}b) If $\omega =0$ in \eqref{semigrestimate}, that is the semigroup
$T_t$ satisfies the inequality
$\|T_t\|\le \nolinebreak M,$ then the condition
$supp \, \mu \subset (0,1)$ in Theorem \ref{continuum} can be
replaced by $supp \, \mu \subset [0,1).$
\end{remark}

Suppose that $L_t$ is an $(\mathcal{F}_t)$-adapted L\'evy process. Let
$E_t$ be a continuous $(\mathcal{F}_t)$-time-change.  Consider the
following SDE driven by the time-changed L\'evy process $L_{E_t}$:
\begin{align} \label{sde00}
   X_t&=x + \int_0^t b(E_{s},X_{s-}) dE_s + \int_0^t \sigma(E_{s}, X_{s-}) dB_{E_s}\\
    + &\int_0^t \hspace{-0.5mm}\int_{|w|<1} \hspace{-2.5mm}H(E_{s},X_{s-}, w)\tilde{N}(dE_s, dw)+
        \int_0^t \hspace{-0.5mm}\int_{|w|\ge 1} \hspace{-2.5mm}K(E_{s}, X_{s-}, w) N(dE_s, dw),\notag
\end{align}
with $x\in\mathbb{R}$, and  continuous maps $b(u,y): \mathbb{R}_+ \times
\mathbb{R}^n \to \mathbb{R}^n,$ $\sigma(u,y): \mathbb{R}_+ \times \nolinebreak
\mathbb{R}^n \to \mathbb{R}^{n \times m},$ and
$G(u,y,w)=\chi_{(|w|<1)}(w)H(u,y,w)+\chi_{(|w|\ge 1)}(w)K(u,y,w):$
$\mathbb{R}_+ \times \mathbb{R}^n \times \mathbb{R}^n \to
\mathbb{R}^n$ satisfying the Lipschitz and growth conditions
\eqref{lip}, \eqref{growth} with respect to the variable $y \in
\mathbb{R}^n$, for each fixed $u\ge 0$.

SDE \eqref{sde00} is obtained from an SDE driven by a L\'evy process
upon replacing its driving process $L_t$ by $L_{E_t}.$  Since $L_t$
is an $(\mathcal{F}_t)$-semimartingale, it follows from
Corollary 10.12 of \cite{Jacod}
that  $L_{E_t}$ is an
$({\mathcal{F}}_{E_{t}})$-semimartingale. Thus, (\ref{sde00}) is the
integral form of an SDE driven by an
$({\mathcal{F}}_{E_{t}})$-semimartingale. We use the following
shorthand form for the differential problem given by SDE
(\ref{sde00}):
\begin{equation}
\label{sde10} dX_t=F(E_{t}, X_{t-}) \odot dL_{E_t}, \ X_{0}=x,
\end{equation}
where $F(u,y)=(b(u,y),\sigma(u,y),G(u,y,\cdot))$ indicates the
triple of coefficients controlling the drift, Brownian, and jump
terms, respectively. In the future, the integral form in
(\ref{sde00}) will also be expressed using the symbol $\odot$.

The SDE in (\ref{sde10}) is closely related to the following SDE
driven by the L\'evy process $L_t$:
\begin{equation}
\label{sde20} dY_{\tau}=F(\tau, Y_{\tau-}) \odot dL_{\tau}, \
Y_{0}=x.
\end{equation}

A duality between SDE \eqref{sde10} and SDE \eqref{sde20} exists as
a special case of a general duality result in \cite{Kobayashi}, but
what is required here is given below.

\begin{theorem}
\label{thsde000}
 Let $D_t$ be a c\`adl\`ag $(\mathcal{F}_t)$-adapted strictly increasing
 process, and $E_t$ be its inverse.
\begin{enumerate}
\item[1)] If $Y_{\tau}$ satisfies SDE (\ref{sde20}), then $X_{t}:=Y_{E_{t}}$
satisfies SDE (\ref{sde10}).
\item[2)] If $X_t$ satisfies SDE (\ref{sde10}), then $Y_{\tau}:=X_{D_{\tau}}$
satisfies SDE (\ref{sde20}).
\end{enumerate}
\end{theorem}

\begin{proof}
Since $D_t$ is a strictly increasing $(\mathcal{F}_t)$-adapted
process, its inverse $E_t$ is a continuous
$(\mathcal{F}_t)$-time-change. Suppose that $Y_\tau$ satisfies SDE
(\ref{sde20}) and let $X_t=Y_{E_t}$. Then by
Proposition 10.21 in \cite{Jacod},
\begin{equation}
   X_t=x+\int_0^{E_t} F(s,Y_{s-})\odot dL_s=x+\int_0^t F(E_s,Y_{E(s)-})\odot dL_{E_s}. \label{thsde000-1}
\end{equation}
$X_t$ will satisfy SDE (\ref{sde10}), provided that $X_{s-} = (Y\circ
E)_{s-}$ can replace $Y_{E(s)-}$ in \eqref{thsde000-1}.  The
equality $Y_{E(s)-}=(Y\circ E)_{s-} $ fails only when $s>0$ and $E$
is constant on some closed interval $[s-\varepsilon,s]\subset (0,t]$
with $\varepsilon>0$.
     However, the integrator
     $L\circ E$ on the right-hand side of \eqref{thsde000-1} is constant on this interval.
     Hence, the difference between the two values $Y_{E(s)-}$ and $X_{s-}=(Y\circ E)_{s-}$ does not affect the value of the integral.  Consequently, \eqref{thsde000-1} is valid with $X_{s-}$ in place of
     $Y_{E(s)-}$.  Thus, $X_t$ satisfies SDE (\ref{sde10}),
establishing part 1).
Similarly, part 2) can be proven using Theorem 3.1 in \cite{Kobayashi}, instead of Proposition 10.21 in \cite{Jacod}.   \qed
\end{proof}

\begin{theorem} \emph{(\cite{Applebaum,Situ})} \label{thsde002}
If $F(u,y)$ satisfies the Lipschitz and growth conditions
\eqref{lip} and \eqref{growth}  for each $u\ge 0$, then SDE
(\ref{sde20}) has a unique strong solution with c\`adl\`ag paths. 
\end{theorem}

\begin{corollary} \label{cor_EU}
If $F(u,y)$ satisfies the Lipschitz and growth conditions
\eqref{lip} and 
\eqref{growth}  for each $u\ge 0$, then SDE
(\ref{sde10}) has a unique 
strong solution with c\`adl\`ag paths.
\end{corollary}

\begin{proof}
A strong solution to SDE \eqref{sde10} clearly exists by Theorems \ref{thsde000} and \ref{thsde002}.  To prove the uniqueness, notice that the driving process $L_{E_t}$ of SDE \eqref{sde10} is constant on any interval $[D_{s-},D_s]$ and so is the solution $X_t$.  
This implies a unique representation $X_t=X_{D_{E(t)}}=Y_{E_t}$, where
$Y_\tau$ is a unique strong solution to SDE \eqref{sde20}.
\qed
\end{proof}

\begin{theorem}\label{multi}
Let $D_{k,t}$, $k=1,\ldots,N$, be independent stable subordinators of
respective indices $\beta_k\in (0,1)$. Define $D_t=\sum_{k=1}^N c_k
D_{k,t},$ with positive constants $c_k$, and let $E_t$ be its
inverse. Suppose that a stochastic process $Y_{\tau}$ satisfies the
SDE \eqref{sdeLevy} driven by a L\'evy process,
where continuous mappings $b, \, \sigma, \, H, \, K$ are bounded and
satisfy condition \eqref{lip}.
 Let $X_t=Y_{E_t}$.  Then
\begin{enumerate}
\item[1)]
$X_t$ satisfies the SDE driven by the time-changed L\'evy process
\begin{align} \label{sdelevyfr}
     X_t&=x + \int_0^t b(X_{s-}) dE_s + \int_0^t \sigma(X_{s-}) dB_{E_s}\\
              + &\int_0^t \hspace{-0.5mm} \int_{|w|<1} \hspace{-2.5mm}H(X_{s-},w)\tilde{N}(dE_s, dw)+ \int_0^t \hspace{-0.5mm}\int_{|w|\ge 1} \hspace{-2.5mm}K(X_{s-},w) N(dE_s, dw);\notag
\end{align}
\item[2)]
if $Y_\tau$ is independent of $E_t$, then the function
$u(t,x)=\mathbb{E}[\varphi(X_t)|X_{0}=x]$ satisfies the following
Cauchy problem
\begin{align} \label{theoremFP}
\hspace{-1.9mm}\sum_{k=1}^N C_k \mathbf{D}_{\ast}^{\beta_k}u(t,x)
=\mathcal{L}(x,\mathbf{D}_x) u(t,x), \ u(0,x)=\varphi(x), \ t>0, \, x \in \mathbb{R}^n,
\end{align}
where $\varphi \in C_0^2(\mathbb{R}^n),$  $C_k=c_k^{{\beta_k}}$,
$k=1, \ldots, N$, and the pseudo-differential operator
$\mathcal{L}(x,\mathbf{D}_x)$ is as in (\ref{levyPsdo}) with symbol
in (\ref{levysymbol}).
\end{enumerate}
\end{theorem}

\begin{proof}
Again, for simplicity, we give the proof in the case $N=2$.

1) Since $D_t$ is a linear combination of stable subordinators,
which are c\`adl\`ag and strictly increasing, it follows that $D_t$
is also c\`adl\`ag and strictly increasing. Hence, it follows from Theorem \ref{thsde000}
that $X_t=Y_{E_t}$
satisfies SDE (\ref{sdelevyfr}).

2) Consider $T_{\tau}^Y\hspace{-0.5mm}
\varphi(x)=\mathbb{E}[\varphi(Y_{\tau})|Y_0=x],$ where $Y_{\tau}$ is
a solution of
SDE (\ref{sdeLevy}).
Then $T^Y_{\tau}$ is a strongly
continuous contraction semigroup in the Banach space
$C_0(\mathbb{R}^n)$ (see \cite{Applebaum})  which satisfies
\eqref{semigrestimate} with $\omega=0,$ has infinitesimal generator
given by the pseudo-differential operator
$\mathcal{L}(x,\mathbf{D}_x)$ with symbol $\Psi(x,\xi)$ defined in
(\ref{levysymbol}), and $C_0^2(\mathbb{R}^n)\subset
\mbox{Dom}(\mathcal{L}(x,\mathbf{D}_x))$. So the function
$p^Y\hspace{-0.5mm}(\tau,x)=T_{\tau}^Y\hspace{-0.5mm} \varphi(x)$ with $\varphi\in
C^2_0(\mathbb{R}^n)$ satisfies the Cauchy problem
\begin{equation}
\label{Cauchy20} \displaystyle{ \frac{\partial p^Y\hspace{-0.5mm}(\tau,x)}{\partial
\tau} = \mathcal{L}(x,\mathbf{D}_x) p^Y\hspace{-0.5mm}(\tau,x)}, ~~ p^Y\hspace{-0.5mm}(0,x) =
\varphi(x).
\end{equation}
Furthermore, consider
$p^X\hspace{-0.5mm}(t,x)=\mathbb{E}[\varphi(X_t)|X_0=x]=\mathbb{E}[\varphi(Y_{E_t})|Y_0=x]$
(recall that $E_0=0$). Using the independence of the processes
$Y_{\tau}$ and $E_t,$
\begin{align}\label{Cauchy25}
p^X\hspace{-0.5mm}(t,x)=\int_0^{\infty}\hspace{-1.5mm}\mathbb{E}[\varphi(Y_{\tau})|E_t=\tau,Y_0=x]f_{E_t}\hspace{-0.5mm}(\tau)d
\tau = \int_0^{\infty} \hspace{-2mm}f_{E_t}\hspace{-0.5mm}(\tau) T^Y_{\tau}\hspace{-0.5mm} \varphi(x) d\tau.
\end{align}
Now, in accordance with Theorem \ref{part2}, $p^X\hspace{-0.5mm}(t,x)$ satisfies
the Cauchy problem (\ref{theoremFP}).
\qed
\end{proof}

\begin{theorem} \label{dodetheorem}
Assume that $D(\mu;t) \in \mathcal{S}$, where $\mu$ is a positive
finite measure with $supp \, \mu \subset [0,1),$  and let $E_t$ be
its inverse. Suppose that a stochastic process $Y_{\tau}$ satisfies
SDE (\ref{sdeLevy}),
and let $X_t=Y_{E_t}$.  Then
\begin{enumerate}
\item[1)]
$X_t$ satisfies SDE (\ref{sdelevyfr});
\item[2)]
if $Y_\tau$ is independent of $E_t$, then the function
$u(t,x)=\mathbb{E}[\varphi(X_t)|X_{0}=x]$ satisfies the following
Cauchy problem
\begin{equation} \label{theoremDODE}
\mathcal{D}_{\mu} u(t,x)=\mathcal{L}(x,\mathbf{D}_x) u(t,x), \ u(0,x)=\varphi(x), \ t>0,
\, x \in \mathbb{R}^n.
\end{equation}
\end{enumerate}
\end{theorem}

\begin{proof}
The proof of part 1) again follows from Theorem \ref{thsde000}.
Part 2) follows from Theorem \ref{continuum} in a manner similar to
the proof of part 2) of Theorem \ref{multi}.
\qed
\end{proof}

\begin{remark}
Theorems \ref{multi} and \ref{dodetheorem} reveal the class of SDEs
which are associated with the wide class of
DODE pseudo-differential equations.
Each SDE in this class is
driven by a semimartingale which is a time-changed L\'evy process,
where the time-change is given by the inverse of a mixture of
independent stable subordinators.  Therefore, these SDEs cannot be
represented as classical SDEs driven by a Brownian motion or a
L\'evy process.
\end{remark}

\begin{corollary}
Let  the coefficients $b, \, \sigma, \, H, \, K$ of  the
pseudo-differential operator $\mathcal{L}(x,\mathbf{D}_x)$ defined
in (\ref{levyPsdo}) with symbol in (\ref{levysymbol}) be continuous,
bounded, and satisfying condition \eqref{lip}. Suppose $\varphi \in
C_0^2(\mathbb{R}^n).$ Then the Cauchy problem
\begin{align}
\mathcal{D}_{\mu} u(t,x)=\mathcal{L}(x,\mathbf{D}_x) u(t,x), \ u(0,x)=\varphi(x), \ t>0,
\, x \in \mathbb{R}^n,\notag
\end{align}
has a unique solution
such that
$u(t,x) \in C_0^2(\mathbb{R}^n)$ for each
$t>0.$
\end{corollary}

\begin{proof}
   The result follows from representation \eqref{Cauchy25} in conjunction with estimate \eqref{estimate000}. \qed
\end{proof}

\noindent
\textit{Example 1. Time-changed $\alpha$-stable L\'evy process.}

Let $L_{\alpha,t}$ be a symmetric $n$-dimensional $\alpha$-stable
L\'evy process, which is a pure jump process. If
$p^L(t,x)=E[\varphi(L_{\alpha,t})|L_{\alpha,\tau}=x],$ where
$\varphi \in C_0^2(\mathbb{R}^n)$,
then $p^L(t,x)$ satisfies in the strong sense the Cauchy problem
(\cite{SDELevy})
\begin{align}
\frac{\partial p^L(t,x)}{\partial t}= - \kappa_{\alpha}
(-\Delta)^{{\alpha}/{2}} p^L(t,x), \ p^L(0,x)=\varphi(x), \ t>0, \, x \in \mathbb{R}^n, \label{levymotion1}
\end{align}
where $\kappa_{\alpha}$ is a constant depending on $\alpha$, and
$(-\Delta)^{\alpha/2}$ is a fractional power of the Laplace
operator. The operator on the right-hand side of \eqref{levymotion1}
can be represented as a pseudo-differential operator with the symbol
$\psi(\xi) := |\xi|^{\alpha}.$

Now suppose that $Y_t$ solves the SDE
 \begin{equation}\label{simple}
dY_t=g(Y_{t-})dL_{\alpha,t}, \ Y_0=x,
\end{equation}
where $g(x)$ is a Lipschitz-continuous function.  Notice that
\eqref{simple} takes the form given in \eqref{sde20} with the pure
jump process $L_{\alpha,t}$ as the driving process and
$F(x)=(0,0,g(x))$. In this case, the forward Kolmogorov equation
takes the form (\cite{SDELevy})
\begin{equation}\label{EX2-1}
\frac{\partial p^Y(t,x)}{\partial t}= - \kappa_{\alpha}
(-\Delta)^{{\alpha}/{2}} \{[g(x)]^{\alpha}p^Y(t,x)\}, ~ t>0, \ x \in
\mathbb{R}^n.
\end{equation}
Application of Theorem \ref{dodetheorem} implies that $X_t=Y_{E_t}$
satisfies the SDE
\begin{align}\label{EX2-3}
dX_t=g(X_{t-})dL_{\alpha,E_t}, \ X_0=x,
\end{align}
where $E_t$ is the first hitting time of the process $D(\mu; t)$
described in this theorem.  Moreover, if $E_t$ is independent of
$Y_t$, then the corresponding forward Kolmogorov equation becomes
\begin{equation}\label{EX2-2}
\mathcal{D}_{\mu} p^X(t,x) = - \kappa_{\alpha}
(-\Delta)^{{\alpha}/{2}} \{[g(x)]^{\alpha}p^X(t,x)\}, \ t>0, \ x \in
\mathbb{R}^n,
\end{equation}
where $\mathcal{D}_{\mu}$ is the operator defined in
\eqref{abstract11111}. When the SDE in  (\ref{EX2-3}) is driven by a
nonsymmetric $\alpha$-stable L\'evy process, an analogue of
(\ref{EX2-2}) holds using instead of (\ref{EX2-1}) its analogue
appearing in \cite{SDELevy}. \qed

\vspace{3mm}

\noindent
\textit{Example 2. Fractional analogue of the Feynman-Kac formula.}

Suppose that $Y_t$ is a strong solution of SDE (\ref{sdeLevy}). Let
$\bar{Y} \in \mathbb{R}^n$ be a fixed point and $q$ be a nonnegative
continuous function. Consider the process  $Y_t^q$ defined by $Y_t^q= Y_t$ {if} $0\le t
< \mathcal{T}_q,$ and $Y_t^q=\bar{Y}$  if $t \ge \mathcal{T}_q,$
where $\mathcal{T}_q$ is an $(\mathcal{F}_t)$-stopping time
satisfying
$\mathbb{P}(\mathcal{T}_q >
t|\mathcal{F}_t)=\exp\bigl(-\int_0^tq(Y_s)ds\bigr).$
Then
$Y^q_t$ is a Feller process with associated semigroup
(see \cite{Applebaum})
\begin{equation}\label{frac-F-K}
(T_t^q
\varphi)(y)=\mathbb{E}\left[\exp\left(-\int_0^tq(Y_s)ds\right)
\varphi(Y_t) \Big| Y_0=y\right],
\end{equation}
and infinitesimal generator $\mathcal{L}_q(x,\mathbf{D}_x)=-q(x)+
\mathcal{L}(x,\mathbf{D}_x),$ where $\mathcal{L}(x,\mathbf{D}_x)$ is
the pseudo-differential operator defined in (\ref{levyPsdo}). Let
$E_t$ be the inverse to a $\beta$-stable subordinator independent of
$Y_t$.  Then it follows from Theorem \ref{multi} with $N=1$ that the
transition probabilities of the process $X_t=Y_{E_t}$ solve the
Cauchy problem for the fractional order equation
\begin{align*}
\mathbf{D}_{\ast}^{\beta} u(t,x)=[-q(x)+\mathcal{L}(x,\mathbf{D}_x)]u(t,x), \ u(0,x)=\varphi(x), \ t > 0, \, x \in \mathbb{R}^n.
\end{align*}
Consequently, \eqref{frac-F-K}, with $X_t=Y_{E_t}$ replacing $Y_t$,
represents a fractional analogue of the Feynman-Kac formula. \qed

\section{Time-changed It\^o formula and its application} \label{sec ALTANATIVE}

This section illustrates a new method of derivation of
time-fractional differential equations based on the time-changed
It\^o formula
{in \cite{Kobayashi}
\textit{without using
the duality principle} (Theorem \ref{thsde000}). For simplicity, we
consider only the one-dimensional case with a Brownian motion as the
driving process.
Let $A^{*}$ be the operator defined as
\begin{equation}\label{conjugateBM}
A^{*}h(y)=-\dfrac\partial {\partial y}\{ b(y) h(y) \}
          + \dfrac 12 \dfrac{\partial^2} {\partial y^2} \{\sigma^2(y) h(y)\}.
\end{equation}

\begin{theorem}\label{Thm FP}\par
     Let $B_t$ be a one-dimensional standard $(\mathcal{F}_t)$-Brownian motion.  Let $D_t=$ $\sum_{k=1}^{N}c_k D_{k,t},$ where $c_k$ are positive constants and
     $D_{k,t}$ are stable
     subordinators of respective indices $\beta_k \in(0,1)$.
     Let $E_t$ be the inverse process to $D_t$.
     Suppose that $X_t$ is a process defined by the SDE
     \begin{align}
          dX_t= b(X_t)dE_t+\sigma(X_t)dB_{E_t}, \ X_0=x, \label{state FP1}
     \end{align}
     where $b(y)$ and $\sigma(y)$ satisfy the Lipschitz condition
     \eqref{lip}.  Suppose also that $X_{D_t}$ is independent of $E_t.$
     Then the transition probability
     $p^X(t,y|x)\equiv p^X(t,y)$
     satisfies in the weak sense the
     time-fractional differential equation
     \begin{align}
\sum_{k=1}^N c_k^{\beta_k} \mathbf{D}_\ast^{\beta_k} p^{X}(t,y) = A^{*} p^X(t,y), 
\label{state FP2}
     \end{align}
     with initial condition $p^X(0,y)=\delta_x(y),$
     the Dirac delta function with mass on $x$,
    where $A^\ast$ is the operator in \eqref{conjugateBM}.
\end{theorem}

\begin{proof}
     For simplicity, the proof is done for $N=2$.   
Let $Y_t=X_{D_t}$.  Then it follows that 
$X_t=X_{D_{E(t)}}=Y_{E_t}$, as in the proof of Corollary \ref{cor_EU}.
Hence, by the independence assumption
between $Y_t$ and $E_t$,
\begin{align} \label{prove FP-6}
p^X(t,y)=\int_0^{\infty} p^Y(u,y)f_{E_t}(u)du
\end{align}
in the sense of distributions.

Since we are not assuming the duality principle (Theorem
\ref{thsde000}), the fact that $p^Y$ satisfies the classical
Kolmogorov equation
$\frac{\partial}{\partial t} p^Y(t,y)=A^\ast
p^Y(t,y)$
cannot be used here.  Instead, we employ the
time-changed It\^o formula to obtain another representation of $p^X$
in terms of $p^Y$ as follows. Let $f\in C_c^\infty(\mathbb{R})$. Since X is constant on every interval $[D_{s-},D_s]$, it follows that $X_{D(s-)}=X_{D_s}=Y_s$
and
the time-changed It\^o formula in \cite{Kobayashi} yields
     \begin{align}
          f(X_t)-f(x)
          &=\int_0^{E_t} f'(Y_s) b(Y_s) ds+ \int_0^{E_t} f'(Y_s) \sigma(Y_s) dB_s  \label{prove FP-1}\\
          & \ \ \ \ \ + \dfrac 12 \int_0^{E_t} f''(Y_s) \sigma^2(Y_s) ds.  \notag
     \end{align}
Because $f\in C_c^\infty(\mathbb{R})$, the process $M$ defined by
$M_u:=\int_0^u f'(Y_s) \sigma(Y_s) dB_s$ is an
$(\mathcal{F}_t)$-martingale.    Taking expectations in \eqref{prove
FP-1} and conditioning on $E_t$ which has density $f_{E_t}$ given in
\eqref{density10}, we have
     \begin{align*}
          &\mathbb{E}[f(X_t)|X_0=x]-f(x)\\
          &=
\int_0^\infty \hspace{-1.8mm}\mathbb{E}\Bigl[M_{u}\hspace{-0.07mm}+\hspace{-0.07mm}\int_0^{u} \hspace{-1.3mm}\Big\{f'(Y_s) b(Y_s) + \dfrac 12 f''(Y_s) \sigma^2(Y_s)
          \Bigr\} ds\Big| E_t=u, Y_0=x\Bigr] f_{E_t}\hspace{-0.5mm}(u)du  \\
          &= \int_0^\infty \hspace{-0.5mm}\int_0^u \hspace{-0.5mm}\mathbb{E}\Bigl[f'(Y_s) b(Y_s)+\dfrac 12f''(Y_s) \sigma^2(Y_s)\Big|Y_0=x\Bigr] ds\; f_{E_t}\hspace{-0.5mm}(u)du
     \end{align*}
     by the assumption that $Y_t=X_{D_t}$ is independent of $E_t$.
The Fubini theorem is allowed since $f\in C_c^\infty(\mathbb{R})$ and $b$ and $\sigma$ are continuous functions.
Using $p^Y$, the above can be rewritten as
     \begin{align}
          &\mathbb{E}[f(X_t)|X_0=x]-f(x)\label{prove FP-2}\\
          &= \int_0^\infty \int_0^u \int_{-\infty}^\infty \Bigl\{f'(y)b(y)+ \dfrac 12 f''(y) \sigma^2(y)\Big\}
p^Y(s,y)
dy\; ds\; f_{E_t}(u)du \notag\\
          &= \int_{-\infty}^\infty f(y) \Bigl\{ \int_0^\infty (J{A}^{*}
p^Y(u,y)
) f_{E_t}(u) du\Bigr\} dy, \notag
     \end{align}
where $J$ is the integral operator. On the other hand, reexpressing
the left-hand side of \eqref{prove FP-2} in terms of $p^X$ yields
     \begin{align}
          \mathbb{E}[f(X_t)|X_0=x]-f(x)=\int_{-\infty}^\infty f(y)
p^X(t,y)
 dy - f(x). \label{prove FP-4}
     \end{align}
     Since $f\in C_c^\infty(\mathbb{R})$ is arbitrary and $C_c^\infty(\mathbb{R})$ is dense in $L^2(\mathbb{R})$,
comparison of \eqref{prove FP-2} and \eqref{prove FP-4}
     leads to another representation of $p^X$ with respect to $p^Y$:
     \begin{align}
         p^X(t,y) -\delta_x(y)= \int_0^\infty (J{A}^{*}p^Y(u,y)) f_{E_t}(u)du \label{prove FP-5}
     \end{align}
     in the sense of distributions with
$p^X(0,y)=\delta_x(y)$.

Now, we use the two representations \eqref{prove FP-6} and
\eqref{prove FP-5} to derive equation \eqref{state FP2} with the
help of Laplace transforms. The Laplace transform of a function
$v(t)$ of the form in \eqref{def_v(t)}, with $f_{E_t}$ in
\eqref{density10}, is computed as in \eqref{laplace_v(t)}. Using
this fact and taking the Laplace transform of both sides in
\eqref{prove FP-6}, we obtain
\begin{align*}
\widetilde{p^X}(s,y)&=(C_1 s^{\beta_1-1}+ C_2
s^{\beta_2-1})\hspace{0.5mm}\widetilde{p^Y}(C_1 s^{\beta_1}+C_2
s^{\beta_2},y), \ s>0,
\end{align*}
where $C_k=c_k^{\beta_k} (k=1,2)$; whereas the Laplace transform of
\eqref{prove FP-5} is
\begin{align*}
 \widetilde{p^X}(s,y)-\frac 1 s \delta_x(y)
&=(C_1 s^{\beta_1-1}+ C_2
s^{\beta_2-1})\hspace{0.5mm}\widetilde{JA^*p^Y}(C_1
s^{\beta_1}+C_2 s^{\beta_2},y)\\
&=\frac {C_1 s^{\beta_1-1}+ C_2 s^{\beta_2-1}}{C_1 s^{\beta_1}+C_2
s^{\beta_2}} \hspace{1mm} \widetilde{A^*p^Y}(C_1 s^{\beta_1}+C_2
s^{\beta_2},y), \ s>0.
\end{align*}
Combining these two identities,
for $s>0$, we have
\begin{align*}
   &C_1 \bigl(s^{\beta_1}\widetilde{p^X}(s,y)-s^{\beta_1-1} \delta_x(y)\bigr)
   +C_2 \bigl(s^{\beta_2}\widetilde{p^X}(s,y)-s^{\beta_2-1} \delta_x(y)\bigr)\\
  &=(C_1 s^{\beta_1}+C_2 s^{\beta_2}) \Bigl( \widetilde{p^X}(s,y)-\frac 1 s \delta_x(y)\Bigr)\\
  &=(C_1 s^{\beta_1-1}+C_2 s^{\beta_2-1})\hspace{0.5mm}\widetilde{A^*p^Y}(C_1
s^{\beta_1}+C_2 s^{\beta_2},y)=\widetilde{A^*p^X}(s,y),
\end{align*}
which coincides with the identity obtained from applying the Laplace
transform to both sides of \eqref{state FP2}.  \qed
\end{proof}

\begin{remark}
If SDE
\eqref{state FP1} contains an additional term $\rho(\hspace{-0.5pt}X_t\hspace{-0.5pt}) dt$, then the method used in the proof of Theorem \ref{Thm FP} does not work since the relationship $Y_{E_t}=X_t$ does not always follow.
Example 5.4 in \cite{Kobayashi}  yields the following conjecture:
if an additional term $\rho(\hspace{-0.5pt}X_t\hspace{-0.5pt})dt$ is included in SDE \eqref{state
FP1}, where $ \rho(y)$ also satisfies the Lipschitz condition, then
it is expected that the partial differential equation corresponding
to \eqref{state FP2} may involve a fractional integral term.
\end{remark}

\begin{acknowledgements}
The authors are indebted to Rudolf Gorenflo, Meredith Burr, Jamison
Wolf, and Xinxin Jiang for references and helpful comments.
We also appreciate suggestions of an
anonymous referee that resulted in a succinct paper.
\end{acknowledgements}


\begin{thebibliography}{}

\bibitem{Applebaum}
   {Applebaum, D.}  {L\'evy Processes and Stochastic Calculus}. Cambridge University Press (2004).


\bibitem{BWM}
   {Benson, D. A.}, {Wheatcraft, S. W.}, and {Meerschaert, M. M.}
   Application of a fractional advection-dispersion equation.
   {Water Resour. Res.} {36(6)} 1403--1412 (2000).


\bibitem{Courrege}
   {Courr\'ege, P.}  G\'en\'erateur infinit\'esimal
   d'un semi-groupe de convolution sur {$\mathbb{R}^n$}, et formule de
   L\'evy-Khinchine. {Bull. Sci. Math. (2)}  {88} 3--30 (1964).

\bibitem{EngelNagel}
    {Engel, K-J.} and {Nagel, R.} {One-parameter Semigroups for Linear Evolution Equations}. Springer (1999).


\bibitem{GillisWeiss}
   {Gillis J. E.} and {Weiss, G. H.}
   Expected number of distinct sites visited by a random walk with an infinite variance.
   { J. Mathematical Phys}. {11} 1307--1312 (1970).

\bibitem{GM97}
    {Gorenflo, R.} and {Mainardi, F.}  Fractional calculus: integral and
    differential equations of fractional order.  In
      A. Carpinteri and F. Mainardi (editors):
    {Fractals and
    Fractional Calculus in Continuum Mechanics}.
    Springer. 223--276 (1997).

\bibitem{GM1}
    {Gorenflo, R.} and {Mainardi, F.}
    Random walk models for space-fractional diffusion processes.
    {Fract. Calc. Appl. Anal.} {1}, ({2}), 167--191 (1998).


\bibitem{GMSR}
   {Gorenflo, R.}, {Mainardi, F.}, {Scalas, E.} and {Raberto, M.}
   Fractional calculus and continuous-time finance. III.
   { Mathematical Finance}.
   171--180.  Trends Math., Birkh\'auser, Basel (2001).

\bibitem{GMV}
   {Gorenflo, R.}, {Mainardi, F.} and {Vivoli, A.}
   Continuous time random walk and parametric subordination in fractional diffusion.
   {Chaos, Solitons Fractals}. {34}, ({1}), 87--103 (2007).


\bibitem{Hermander}
    {H\"ormander, L.}
    {The Analysis of Linear Partial Differential Operators. II. Differential Operators with Constant Coefficients}.
    Springer-Verlag, Berlin (1983).

\bibitem{Jacob}
   {Jacob, N.}
   {Pseudo Differential Operators and Markov Processes. Vol. II. Generators and their Potential Theory}.
   Imperial College Press, London (2002).

\bibitem{Jacod}
   {Jacod, J.}  {Calcul Stochastique et Probl\`emes de Martingales}.
   Lecture Notes in Mathematics, {714}. Springer, Berlin (1979).


\bibitem{Kobayashi}
{Kobayashi, K.}  {Stochastic calculus for a time-changed
semimartingale and the associated stochastic differential
equations.} arXiv:0906.5385v1 [math.PR] (2009).

\bibitem{MagdziarzWeron}
{Magdziarz, M.} and {Weron, A.}  Competition between subdiffusion
and L\'evy flights: a Monte Carlo approach. {Phys. Rev. E} {75},
056702 (2007).

\bibitem{MagdziarzWeronKlafter}
{Magdziarz, M.}, {Weron, A.} and {Klafter, J.}  Equivalence of the
fractional Fokker-Planck and subordinated Langevin equations: the
case of a time-dependent force. {Phys. Rev. Lett.} {101}, 210601
(2008).

\bibitem{MinardiLuchkoPagnini}
     {Mainardi, F.}, {Luchko, Y.}, and {Pagnini, G.}
     The fundamental solution of
     the space-time fractional diffusion equation.
     {Fract. Calc. Appl. Anal.} {4}, ({2}), 153--192 (2001).

\bibitem{MeerschaertScheffler}
     {Meerschaert, M. M.} and {Scheffler, H-P.}
     {Limit Distributions for Sums of Independent Random Vectors.
     Heavy Tails in Theory and Practice}. John Wiley and Sons, Inc. (2001).

\bibitem{Meeschaert1}
     {Meerschaert, M. M.} and {Scheffler, H-P.}  Triangular array limits for continuous time random walks.
     {Stochastic Process. Appl.} { 118}, 1606--1633 (2008).

\bibitem{MetzlerKlafter00}
     {Metzler, R.} and {Klafter, J.}
     The random walk's guide to anomalous diffusion: a fractional dynamics approach.
     {Phys. Rep}. {339}, ({1}), 1--77 (2000).

\bibitem{MontrollScher}
      {Montroll E. W.} and {Weiss G. H.}
      Random walks on lattices. II.
      {J. Mathematical Phys.} {6}, 167--181 (1965).


\bibitem{Sato}
    {Sato, K-i.}
    {L\'evy Processes and Infinitely Divisible Distributions}.
    Cambridge University Press (1999).

\bibitem{Saxton}
   {Saxton, M. J.} and {Jacobson, K.}
   Single-particle tracking: applications to membrane dynamics.
   {Annu. Rev. Biophys. Biomol. Struct.} {26}, 373--399 (1997).

\bibitem{SDELevy}
   {Schertzer, D.}, {Larchev\^eque, M.}, {Duan, J.}, {Yanovsky, V. V.}, and {Lovejoy, S.}
   Fractional Fokker-Planck equation for nonlinear stochastic
   differential equations driven by non-Gaussian {L\'evy} stable noises.
   {J. Math. Phys.} {42}, ({1}), 200--212 (2001).

\bibitem{Situ}
    {Situ, R.}  {Theory of Stochastic Differential Equations with Jumps and Applications:
    Mathematical and Analytical Techniques with Applications to
    Engineering}, Springer (2005).


\bibitem{Taylor}
     {Taylor, M.}
     {Pseudodifferential Operators}. Princeton University Press (1981).

\bibitem{UzhaykinZolotarev}
  {Uchaikin, V. V.} and {Zolotarev, V. M.}
  {Chance and Stability. Stable Distributions and their Applications}.
  VSP, Utrecht (1999).

\bibitem{UG}
   {Umarov, S.} and {Gorenflo, R.}
   On multi-dimensional random walk models approximating symmetric space-fractional diffusion processes.
   {Fract. Calc. Appl. Anal.} {8}, ({1}), 73--88 (2005).

\bibitem{UG05}
   {Umarov, S.} and {Gorenflo, R.}
   Cauchy and nonlocal multi-point problems for distributed order pseudo-differential equations. I.
   {Z. Anal. Anwendungen} {24}, ({3}), 449--466 (2005).

\bibitem{US}
   {Umarov, S.} and {Steinberg, S.}
   Random walk models associated with distributed fractional order differential equations.
   {IMS Lecture Notes Monogr. Ser. High Dimensional Probability.}
   {51}, 117--127 (2006).

\bibitem{Widder}
   {Widder, D. V.}
   {The Laplace transform}.
   Princeton University Press (1941).

\bibitem{Zaslavski}
   {Zaslavsky, G. M.}
   Chaos, fractional kinetics, and anomalous transport.
   {Phys. Rep.} {371}, ({6}), 461--580 (2002).

\end{thebibliography}
\end{document}